\documentclass[3p,times]{elsarticle}

\usepackage{amsmath}
\usepackage{amsthm}
\usepackage{amsfonts}

\newtheorem{theorem}{Theorem}
\newtheorem{definition}[theorem]{Definition}

\newtheorem{lemma}[theorem]{Lemma}
\newtheorem{example}[theorem]{Example}
\newtheorem{corollary}[theorem]{Corollary}

\newcommand{\N}{\mathbb N}

\begin{document}

\begin{frontmatter}

\date{}	
\title{On a certain generalization of $W$-spaces}
\author[add1,add2]{Martin Dole\v zal\fnref{thanks1}}
\ead{dolezal@math.cas.cz}
\author[add3]{Warren~B.~Moors}
\ead{w.moors@auckland.ac.nz}
\address[add1]{Institute of Mathematics, University of Warsaw, Banacha 2, 02-097 Warszawa,	Poland}
\address[add2]{Institute of Mathematics of the Czech Academy of Sciences. \v Zitn\'a 25, 115 67 Praha 1, Czech Republic}
\address[add3]{Department of Mathematics, The University of Auckland, Private Bag 92019, Auckland Mail Centre, Auckland 1142, New Zealand}
\fntext[thanks1]{Research of Martin Dole\v zal was supported by the GA\v CR project 17-27844S and by the FP7-PEOPLE-2012-IRSES project AOS(318910), and by RVO: 67985840.}

\begin{abstract}
We present a simple generalization of $W$-spaces introduced by G.~Gruenhage. We show that this generalization leads to a strictly larger class of topological spaces which we call $\widetilde W$-spaces, and we provide several applications. Namely, we use the notion of $\widetilde W$-spaces to provide sufficient conditions for the product of two spaces to be a Baire space, for a semitopological group to be a topological group, or for a separately continuous function to be continuous at the points of a certain large set.
\end{abstract}

			
\end{frontmatter}
	
\section{Introduction}\label{chapter:introduction}
All topological spaces considered in this paper are non-empty and Hausdorff.
	
One of the possible generalizations of first countable spaces are $W$-spaces which were introduces by G.~Gruenhage in~\cite{Gruenhage1}, and studied in detail by the same author in \cite{Gruenhage2}. $W$-spaces are defined in terms of the following topological game. Let $x$ be a point of a topological space $(X,\tau)$. Player~I chooses an open set $U_1$ containing $x$, then player~II chooses a point $x_1\in U_1$. Next, player~I chooses an open set $U_2$ containing $x$, and player~II chooses a point $x_2\in U_2$, and so on. Player~I wins if $x$ is an accumulation point of the sequence $(x_n)_{n\in\N}$. Otherwise player~II wins. We denote this topological game by $G(x)$.	We say that a point $x$ of a topological space $(X,\tau)$ is a $W$-point if player~I has a winning strategy in the game $G(x)$. Finally, we say that a topological space $(X,\tau)$ is a $W$-space if every point $x\in X$ is a $W$-point. (Note that the original winning condition for player~I in the game $G(x)$ was different in \cite{Gruenhage1}; it was required that the sequence $(x_n)_{n\in\N}$ converges to $x$. However, it follows from \cite[Theorem 3.9]{Gruenhage2} that the winning condition can be relaxed to the one which we use, without altering the property of being a $W$-point.)
	
The paper \cite{Gruenhage2} mainly investigates the relationship of $W$-spaces to some other classes of spaces. As an example, let us recall that every first-countable space is a $W$-space, and every $W$-space is countably bi-sequential.
	
In this paper, we provide a further generalization of $W$-spaces (and thus of first-countable spaces) which we call $\widetilde W$-spaces. To do this, we introduce a new topological game which is very similar to the game $G(x)$ described above. But we do not require player~I to choose the open sets such that they contain $x$. Instead, player~I may choose arbitrary non-empty open sets. At first sight, this may look a little weird as one would expect that any reasonable generalization of first-countability should deal with `neighborhoods'. However, we provide several applications which suggest that this new notion could be useful. First, we prove Theorem~\ref{th:productOfBaireSpaces} which is an improvement of~\cite[Theorem~4.4]{Lin_Moors}. This theorem is more general than the statement that the product $X\times Y$ of a Baire space $(X,\tau)$ and a hereditarily Baire space $(Y,\tau')$ is Baire provided $(Y,\tau')$ is a $\widetilde W$-space. (The proof of Theorem~\ref{th:productOfBaireSpaces} is almost the same as the proof of~\cite[Theorem~4.4]{Lin_Moors}, we only realized that it was not used in the proof of~\cite[Theorem~4.4]{Lin_Moors} that the moves of player~I in the game $G(x)$ are neighborhoods of $x$.) In Theorem~\ref{QC-result}, we provide a new sufficient condition for a separately continuous function to be quasi-continuous at certain points. This is a generalization of \cite[Lemma 2]{moors} where points with a countable local base are used instead of $\widetilde W$-points. Theorem~\ref{QC-result} can be used to the study of when a semitopological group is a topological group (see Corollary~\ref{Corollary1} which is a generalization of \cite[Corollary~1]{moors}) or when a separately continuous function is continuous at the points of a certain large set (see Corollary~\ref{Cor:sepCont-cont}). To justify our results, we also provide examples of $\widetilde W$-spaces which are very far from being $W$-spaces (see Examples~\ref{ex:betaN} and \ref{ex:sigma_prod}).

The organization of the paper is very simple. In Chapter~\ref{chapter:notation}, we introduce all the notation needed to prove our results. In Chapter~\ref{chapter:results}, we define the key notions of $\widetilde W$-points and $\widetilde W$-spaces, and we prove our results.
	
\section{Notation}\label{chapter:notation}
First, we recall the definition of $W$-points and $W$-spaces. For a point $x$ of a topological space $(X,\tau)$, we denote by by $G(x)$ the topological game introduced by G.~Gruenhage in \cite{Gruenhage1} (this is the game described in the introduction).

\begin{definition}[G.~Gruenhage]
We say that a point $x$ of a topological space $(X,\tau)$ is a \emph{$W$-point} in $X$ if player~I has a winning strategy in the game $G(x)$.
	
We say that a topological space $(X,\tau)$ is a \emph{$W$-space} if every point $x\in X$ is a $W$-point in $X$.
\end{definition}

By a \emph{strategy} for player~I in the game $G(x)$ we mean a rule that specifies each move of player~I in every possible situation. More precisely, a strategy $\sigma$ for player~I in the game $G(x)$ is a mapping $\sigma\colon X^{<\N}\rightarrow\tau$ defined on the set $X^{<\N}$ of all finite (possibly empty) sequences of elements of $X$ whose values are open subsets of $X$, such that $x\in\sigma(x_1,\ldots,x_n)$ for every $(x_1,\ldots,x_n)\in X^{<\N}$. If player~I follows a strategy $\sigma$ then he starts the game by playing $\sigma(\emptyset)$ in his first move. If player~II replies by choosing some $x_1\in\sigma(\emptyset)$ then player~I plays $\sigma(x_1)$ in his second move. If player~II replies by some $x_2\in\sigma(x_1)$ then player~I continues by $\sigma(x_1,x_2)$, and so on.
A strategy $\sigma$ for player~I in the game $G(x)$ is called a \emph{winning strategy} if player~I wins each run of the game $G(x)$ when following the strategy.
A finite (resp. infinite) sequence $(x_j)_j$ of elements of $X$ is called a \emph{$\sigma$-sequence} if player~II can play his finitely many first moves (resp. each of his moves) of the game $G(x)$ according to the sequence if player~I follows his strategy $\sigma$. That means that $(x_j)_j$ is a $\sigma$-sequence if and only if $x_j\in\sigma(x_1,\ldots,x_{j-1})$ for every $j$.

A (winning) strategy for player~II in the game $G(x)$ can be defined analogously. One can also similarly define (winning) strategies for either player, as well as $\sigma$-sequences (where $\sigma$ is a strategy for either player), in other topological games (in Chapter~\ref{chapter:results}, we will introduce the game $\widetilde G(x)$ and recall the games $BM(R)$ and $\mathcal{G}(X)$).

Suppose that $\{(X_s,\tau_s)\colon s\in S\}$ is a nonempty family of topological spaces and $a=(a_s)_{s\in S}\in\Pi_{s\in S}X_s$. Then the \emph{$\Sigma$-product} $\Sigma_{s\in S}X_s(a)$ of the family $\{(X_s,\tau_s)\colon s\in S\}$ with the \emph{base point} $a$ is the set
\begin{equation*}
\left\{(x_s)_{s\in S}\in\Pi_{s\in S}X_s\colon x_s\neq a_s\text{ for at most countably many }s\in S\right\}
\end{equation*}
endowed with the topology inherited from the product space $\Pi_{s\in S}X_s$.

We also need the notion of a ``rich family'' which was first defined and used in \cite{Borwein_Moors}.
Let $(X,\tau)$ be a topological space and let $\mathcal F$ be a family of nonempty closed separable subspaces of $X$. Then $\mathcal F$ is called a \emph{rich family} if the following two conditions are satisfied:
\begin{itemize}
\item[(i)] for every separable subspace $Y$ of $X$, there is $F\in\mathcal F$ such that $Y\subseteq F$,
\item[(ii)] for every increasing (with respect to inclusion) sequence $(F_n)_{n\in\N}$ of elements of $\mathcal F$ it holds $\overline{\bigcup_{n\in\N}F_n}\in\mathcal F$.
\end{itemize}

Recall that a subset of a topological space is called \emph{meager} (or a \emph{set of first category}) if it is the union of countably many nowhere dense sets. A \emph{comeager} (also called \emph{residual}) set is the complement of a meager set. In other words, a comeager set is the intersection of countably many sets with dense interiors.

A topological space is called \emph{Baire} if every intersection of countably many open dense subsets of the space is dense.

\section{$\widetilde W$-points and $\widetilde W$-spaces}\label{chapter:results}
We define $\widetilde W$-spaces in terms of the following topological game. Let $x$ be a point of a topological space $(X,\tau)$. Player~I chooses a non-empty open set $U_1\subseteq X$, then player~II chooses a point $x_1\in U_1$. Next, player~I chooses a non-empty open set $U_2\subseteq X$, and player~II chooses a point $x_2\in U_2$, and so on. Player~I wins if $x$ is an accumulation point of the sequence $(x_n)_{n\in\N}$. Otherwise player~II wins. We denote this topological game by $\widetilde G(x)$.
	
\begin{definition}
We say that a point $x$ of a topological space $(X,\tau)$ is a \emph{$\widetilde W$-point} in $X$ if player~I has a winning strategy in the game $\widetilde G(x)$.
		
We say that a topological space $(X,\tau)$ is a \emph{$\widetilde W$-space} if every point $x\in X$ is a $\widetilde W$-point in $X$.
\end{definition}
	
\begin{lemma}\label{lem:countable_closure}
Let $(X,\tau)$ be a topological space. Whenever $x\in X$ is in the closure of a countable subset of $X$ consisting of $\widetilde W$-points then $x$ is also a $\widetilde W$-point.
\end{lemma}
	
\begin{proof}
Suppose that $x\in X$ is in the closure of the set $\{y_k\colon k\in\N\}$ where each $y_k$ is a $\widetilde W$-point. We need to show that $x$ is a $\widetilde W$-point. For each $k\in\N$, we fix a winning strategy $\sigma_k$ for player~I in the game $\widetilde G(y_k)$. We fix a sequence $\{k_n\colon n\in\N\}$ of natural numbers such that each $k\in\N$ occurs infinitely many times in the sequence. We define a strategy for player~I in the game $\widetilde G(x)$ as follows. The first move of player~I is $\sigma_{k_1}(\emptyset)$. Now suppose that the first $n$ moves of the game $\widetilde G(x)$ have been played (and so the points $x_1,\ldots,x_n$ chosen by player~II are already known). Let $p_1<\ldots<p_m$ be all natural numbers $1\le p\le n$ for which $k_p=k_{n+1}$. Then in his $(n+1)$th move, player~I chooses $\sigma_{k_{n+1}}(x_{p_1},\ldots,x_{p_m})$. It is easy to see that the described strategy for player~I is winning. Indeed, each $y_k$ is an accumulation point of the subsequence $(x_n)_{n\in\N,\,k_n=k}$ since the strategy $\sigma_k$ is winning for player~I in the game $\widetilde G(y_k)$. Therefore each $y_k$ is also an accumulation point of the sequence $(x_n)_{n\in\N}$. Now the conclusion immediately follows from the fact that the set of all accumulation points of any given sequence is always a closed set.
\end{proof}

\begin{example}\label{ex:betaN}
The Stone-\v Cech compactification $\beta\N$ of natural numbers $\N$ is a $\widetilde W$-space but not a $W$-space.
\end{example}
	
\begin{proof}
The space $\beta\N$ is a regular separable space which is not first-countable, and so it is not a $W$-space (see \cite[Theorem~3.6]{Gruenhage2}).
		
On the other hand, every $n\in\N$ is clearly a $\widetilde W$-point in $\beta\N$ as it is an isolated point. Therefore by Lemma~\ref{lem:countable_closure}, the space $\beta\N$ is a $\widetilde W$-space.
\end{proof}
	
Note that the space $\beta\N$ contains a dense subspace $\N$ consisting of $W$-points. Therefore the following example is even stronger than the previous one. Also note that Theorem~\ref{th:productOfBaireSpaces}, as a generalization of \cite[Theorem~4.4]{Lin_Moors}, is not justified just by Example~\ref{ex:betaN}. Indeed, $X\times\N$ is a Baire space whenever $(X,\tau)$ is a Baire space (this is trivial but it also follows from the statement of \cite[Theorem~4.4]{Lin_Moors}). And as $X\times\N$ is a dense subspace of $X\times\beta\N$, we conclude that $X\times\beta\N$ is also a Baire space (and in fact, this follows from the second part of \cite[Theorem~4.4]{Lin_Moors}). On the other hand, Example~\ref{ex:sigma_prod} justifies Theorem~\ref{th:productOfBaireSpaces} since the topological space constructed in that example contains no dense subspaces which are $W$-spaces.
	
\begin{example}\label{ex:sigma_prod}
There is a topological space $(X,\tau)$ with the following properties:
\begin{itemize}
\item $X$ is a $\widetilde W$-space,
\item whenever $Z$ is a dense subspace of $X$ then no point $x\in Z$ is a $W$-point in $Z$ (in particular, no point $x\in X$ is a $W$-point in $X$),
\item $X$ possesses a rich family of Baire subspaces.
\end{itemize} 
\end{example}
	
\begin{proof}
Let $S$ be an uncountable set. For every $s\in S$, we put $X^s=\beta\N$. For every $a=(a^s)_{s\in S}\in\Pi_{s\in S}\beta\N$, we define a space $X_a$ as the $\Sigma$-product $\Sigma_{s\in S}X^s(a)$ with the base point $a$. First, we prove that each $X_a$ is a $\widetilde W$-space. Let us fix $a=(a^s)_{s\in S}\in\Pi_{s\in S}\beta\N$ and a point $x=(x^s)_{s\in S}\in X_a$, we will construct a strategy $\sigma$ for player~I in the game $\widetilde G(x)$ played in $X_a$. For every $y=(y^s)_{s\in S}\in X_a$, we fix an infinite sequence $(s_1(y),s_2(y),\ldots)$ of elements of $S$ containing all indices $s\in S$ for which $y^s\neq x^s$ (so for $y=x$, this may be an arbitrary sequence of elements of $S$). We also fix a sequence $(t_n)_{n\in\N}$ of infinite sequences of natural numbers such that every finite sequence $t\in\N^{<\N}$ of natural numbers is an initial segment of infinitely many sequences $t_n$. We define the first move of player~I as $\sigma(\emptyset)=X_a$. Let $x_1=(x_1^s)_{s\in S}$ be the first move of player~II. We put $r_1=s_1(x_1)$. Then we define the second move of player~I as the set of all points $(y^s)_{s\in S}\in X_a$ for which $y^{r_1}$ is equal to the first element of the sequence $t_1$, that is
\begin{equation}\nonumber
\sigma(x_1)=\{(y^s)_{s\in S}\in X_a\colon (y^{r_1})\text{ is an initial segment of }t_1\}.
\end{equation}
Let $x_2=(x_2^s)_{s\in S}$ be the second move of player~II. Then we prolong the sequence $(r_1)$ (of length $k_1=1$) to a sequence $(r_1,\ldots,r_{k_2})$ of pairwise distinct indices from the set $S$ (for some natural number $k_2\ge 1$) such that $\{r_1,\ldots,r_{k_2}\}=\{s_i(x_j)\colon i,j\le 2\}$. Now suppose that for some $n\ge 2$, the first $n$ moves of the game $\widetilde G(x)$ have been played. Let $x_i=(x_i^s)_{s\in S}$, $i=1,\ldots,n$, be the first $n$ moves of player~II. Suppose also that we have already defined a sequence $(r_1,\ldots,r_{k_n})$ of pairwise distinct indices from the set $S$ (where $k_n$ is some natural number) such that $\{r_1,\ldots,r_{k_n}\}=\{s_i(x_j)\colon i,j\le n\}$. Then we define the $(n+1)$th move of player~I by
\begin{equation}\nonumber
\sigma(x_1,\ldots,x_n)=\{(y^s)_{s\in S}\in X_a\colon (y^{r_1},\ldots,y^{r_{k_n}})\text{ is an initial segment of }t_n\}.
\end{equation}
Let $x_{n+1}=(x_{n+1}^s)_{s\in S}$ be the $(n+1)$th move of player~II. Then we prolong the sequence $(r_1,\ldots,r_{k_n})$ to a sequence $(r_1,\ldots,r_{k_{n+1}})$ of pairwise distinct indices from the set $S$ (for some natural number $k_{n+1}\ge k_n$) such that $\{r_1,\ldots,r_{k_{n+1}}\}=\{s_i(x_j)\colon i,j\le n+1\}$.	This completes the construction of the strategy $\sigma$. Next, we show that $\sigma$ is a winning strategy for player~I. We fix a natural number $n_0$ and an open neighborhood $U$ of $x$ of the form
\begin{equation}\nonumber
U=\{(y^s)_{s\in S}\in X_a\colon y^{p_m}\in U_m\text{ for }m=1,\ldots,q\}
\end{equation}
for some pairwise distinct indices $p_1,\ldots,p_q$ from $S$ and for some open neighborhoods $U_m$ of $x^{p_m}$ in $\beta\N$, $m=1,\ldots,q$. We need to find a natural number $n\ge n_0$ such that $x_{n+1}\in U$. We may assume that for every $m=1,\ldots,q$, there are $i,j\in\N$ such that $p_m=s_i(x_j)$, and so there is also $a_m\in\N$ such that $p_m=r_{a_m}$. Recall that we constructed a nondecreasing sequence $(k_n)_{n\in\N}$ of natural numbers for which there clearly is $n_1\in\N$ such that $k_n\ge l:=\max\{a_1,\ldots,a_q\}$ for every $n\ge n_1$. For every $1\le m\le q$, we fix some $b_m\in U_m\cap\N$. We also fix a finite sequence $f=(f_1,\ldots,f_l)$ of natural numbers of length $l$ such that $f_{a_m}=b_m$ for every $m=1,\ldots,q$. By the choice of the sequence $(t_n)_{n\in\N}$, there is a natural number $n\ge\max\{n_0,n_1\}$ such that the finite sequence $f$ is an initial segment of $t_n$. Then	
\begin{equation}\nonumber
\begin{split}
\sigma(x_1,\ldots,x_n)=\{&(y^s)_{s\in S}\in X_a\colon(y^{r_1},\ldots,y^{r_{k_n}})\text{ is an initial segment of }t_n\}\\
\subseteq\{&(y^s)_{s\in S}\in X_a\colon(y^{r_1},\ldots,y^{r_l})\text{ is an initial segment of }t_n\}\\
\subseteq\{&(y^s)_{s\in S}\in X_a\colon y^{r_{a_m}}=f_{a_m}\text{ for }m=1,\ldots,q\}\\
=\{&(y^s)_{s\in S}\in X_a\colon y^{p_m}=b_m\text{ for }m=1,\ldots,q\}\subseteq U,
\end{split}
\end{equation}
and so $x_{n+1}\in\sigma(x_1,\ldots,x_n)\subseteq U$. This shows that $X_a$ is a $\widetilde W$-space for every $a\in\Pi_{s\in S}\beta\N$.
	
In the rest of the proof, we use the well known identification of elements of $\beta\N$ with ultrafilters on the set $\N$ of all natural numbers. Let $\mathcal F$ be the filter consisting of all subsets $A$ of $\N$ with
\begin{equation}\nonumber
\liminf_{k\rightarrow\infty}\frac{|\{i\in A\colon 1\le i\le k\}|}{k}=1.
\end{equation}
We fix an ultrafilter $\mathcal U$ on $\N$ containing the filter $\mathcal F$. Now suppose that the point $a=(a^s)_{s\in S}$ from the previous construction is chosen such that each $a^s$, $s\in S$, is the element of $\beta\N$ corresponding to the ultrafilter $\mathcal U$. We will show that whenever $Z$ is a dense subspace of $X:=X_a$ then no point $x\in Z$ is a $W$-point in $Z$. To this end, we fix a dense subspace $Z$ of $X$ and a point $x=(x^s)_{s\in S}\in Z$. We find a coordinate $s_0\in S$ such that $x^{s_0}=a^{s_0}$ (i.e., the point $x^{s_0}\in\beta\N$ corresponds to the ultrafilter $\mathcal U$). We define a strategy for player~II in the game $G(x)$ played in the subspace $Z$ as follows. Suppose that $U_n$ is the $n$th move of player~I (for some $n\in\N$). We find an open subset $W_n$ of $X$ such that $U_n=W_n\cap Z$.	Then the projection $\pi^{s_0}(W_n)$ of $W_n$ to the coordinate $s_0$ is an open neighborhood of $x^{s_0}$ in $\beta\N$, and so it has an infinite intersection with the subset $\N$ of $\beta\N$. Using this fact together with the density of $Z$ in $X$, player~II can choose his $n$th move $x_n=(x_n^s)_{s\in S}\in U_n$ such that $x_n^{s_0}\in\N$ and $x_n^{s_0}\ge n^2$. We show that this strategy is winning for player~II which will complete the proof. It suffices to show that $x^{s_0}$ is not a cluster point of the set $C=\{x_n^{s_0}\colon n\in\N\}$. By the description of the strategy, it holds
\begin{equation}\nonumber
\limsup_{k\rightarrow\infty}\frac{|\{i\in C\colon 1\le i\le k\}|}{k}=0.
\end{equation}
This means that $\N\setminus C\in\mathcal F\subseteq\mathcal U$. Therefore the set
\begin{equation}\nonumber
U=\{y\in\beta\N\colon y\text{ corresponds to an ultrafilter containing }\N\setminus C\}
\end{equation}
is an open neighborhood of $x^{s_0}$ in $\beta\N$. But $U$ does not intersect the set $C$, and so $x^{s_0}$ is not a cluster point of $C$.
	
Finally, the family of all subspaces of $X$ of the form $\{(x^s)_{s\in S}\in X\colon x^s=a^s\text{ for every }s\in S\setminus S'\}$ where $S'$ is an at most countable subset of $S$ is clearly a rich family of Baire subspaces.
\end{proof}

Note that it is not difficult to see that the space $(X,\tau)$ constructed in Example~\ref{ex:sigma_prod} is even a Baire space (as it is a $\Sigma$-product of compact spaces).

Recall that in the Banach-Mazur game $BM(R)$ played in a topological space $(X,\tau)$ with a subset $R$ of $X$, two players $\beta$ (who starts the game) and $\alpha$ alternately construct a decreasing sequence $(B_n)_{n\in\N}$ of nonempty open subsets of $X$. Player~$\beta$ wins the game if $\bigcap_{n\in\N}B_n\not\subseteq R$. Otherwise player~$\alpha$ wins.

\begin{theorem}{\cite{Oxtoby}}\label{th:OxtobyBM}
Let $R$ be a subset of a topological space $(X,\tau)$. Then $R$ is comeager in $X$ if and only if player~$\alpha$ has a winning strategy in the game $BM(R)$.
\end{theorem}

\begin{lemma}{\cite[Lemma~4.2]{Lin_Moors}}\label{lem:find_W}
Let $(X,\tau)$, $(Y,\tau')$ be topological spaces, and let $O$ be an open dense subset of $X\times Y$. Let $U$ be a nonempty open subset of $X$, and let $V_1,\ldots,V_m$ be nonempty open subsets of $Y$. Then there exist a nonempty open subset $W$ of $U$ and points $z_i\in V_i$, $1\le i\le m$, such that $W\times\{z_1,\ldots,z_m\}\subseteq O$.
\end{lemma}

The following lemma has the same proof as \cite[Theorem~4.3]{Lin_Moors}. We only observed that it was not used in the proof of~\cite[Theorem~4.3]{Lin_Moors} that the moves of player~I in the game $G(x)$ are neighborhoods of $x$.

\begin{lemma}\label{lem:productOfBaireSpaces}
Let $(X,\tau)$ be a topological space, and let $(Y,\tau')$ be a $\widetilde W$-space. Let $Z$ be a separable subspace of $Y$, and let $\{O_n\colon n\in\N\}$ be a countable system of dense open subsets of $X\times Y$. Then for every rich family $\mathcal F$ in $Y$, the set
\begin{equation}\nonumber
\begin{split}
R=\big\{x\in X\colon&\text{ there is }F_x\in\mathcal F\text{ containing }Z\text{ such that }\\
&\{y\in F_x\colon(x,y)\in O_n\}\text{ is dense in }F_x\text{ for all }n\in\N\big\}
\end{split}
\end{equation}
is comeager in $X$.
\end{lemma}

\begin{proof}
Let us fix a rich family $\mathcal F$ in $Y$. We may assume that $Y$ is infinite, otherwise the assertion is trivial. Then we may also assume that all elements of $\mathcal F$ are infinite. Moreover, we may assume that the sequence $(O_n)_{n\in\N}$ is decreasing (with respect to inclusion). For every $y\in Y$, we fix a winning strategy $\sigma_y$ for player~I in the game $\widetilde G(y)$ played in $Y$. We will construct a strategy $\sigma$ for player~$\alpha$ in the game $BM(R)$ played in $X$, and then we will show that this strategy is winning. The rest will follow from Theorem~\ref{th:OxtobyBM}.
	
We start the construction of the strategy $\sigma$ by fixing a countable subset $F_1=\{f_{(1,j)}\colon j\in\N\}$ of $Y$ such that $Z\subseteq\overline{F_1}\in \mathcal F$. Let $B_1\subseteq X$ be the first move of player~$\beta$ in the game $BM(R)$. By Lemma~\ref{lem:find_W}, there are a nonempty open subset $W_1$ of $B_1$ and $z_{(1,1,1)}\in\sigma_{f_{(1,1)}}(\emptyset)$ such that $W_1\times\{z_{(1,1,1)}\}\subseteq O_1$. We define $Z_1=\{z_{(1,1,1)}\}$ and $\sigma(B_1)=W_1$. Note that the sequence $(z_{(1,1,1)})$ (of length 1) is a $\sigma_{f_{(1,1)}}$-sequence.
	
Now suppose that player~$\beta$ has already played his first $n$ moves $B_1,\ldots,B_n$ of the game $BM(R)$. Suppose also that for every $1\le k\le n-1$, we have already defined the $k$th move $\sigma(B_1,\ldots,B_k)$ of player~$\alpha$ in the game $BM(R)$, together with a countable subset $F_k=\{f_{(k,j)}\colon j\in\N\}$ of $Y$ and a finite subset $Z_k=\{z_{(i,j,l)}\colon i+j+l\le k+2\}$ of $Y$ such that
\begin{itemize}
\item[(i)] $Z_{k-1}\cup F_{k-1}\subseteq\overline{F_k}\in\mathcal F$ (if $k\ge 2$),
\item[(ii)] the sequence $(z_{(i,j,1)},\ldots,z_{(i,j,l)})$ is a $\sigma_{f_{(i,j)}}$-sequence for every $i,j,l\in\N$ with $i+j+l=k+2$,
\item[(iii)] $\sigma(B_1,\ldots,B_k)\times\{z_{(i,j,l)}\colon i+j+l=k+2\}\subseteq O_k$.
\end{itemize}
Then we find a countable subset $F_n=\{f_{(n,j)}\colon j\in\N\}$ of $Y$ such that $Z_{n-1}\cup F_{n-1}\subseteq\overline{F_n}\in \mathcal F$. By Lemma~\ref{lem:find_W}, there are a nonempty open subset $W_n$ of $B_n$ and points $z_{(i,j,n+2-i-j)}\in\sigma_{f_{(i,j)}}(z_{(i,j,1)},\ldots,z_{(i,j,n+1-i-j)})$ for every $i+j\le n+1$, such that $W_n\times\{z_{(i,j,n+2-i-j)}\colon i+j\le n+1\}\subseteq O_n$. We define $Z_n=\{z_{(i,j,l)}\colon i+j+l\le n+2\}$ and $\sigma(B_1,\ldots,B_n)=W_n$. Note that conditions (i)-(iii) are satisfied for $k=n$. This completes the construction of the strategy $\sigma$.
	
It remains to show that $\sigma$ is a winning strategy for player~$\alpha$. To this end, let $(B_n\colon n\in\N)$ be a $\sigma$-sequence. Let us fix $x\in\bigcap_{n\in\N}B_n$, we need to prove that $x\in R$. The subspace $F_x=\overline{\bigcup_{n\in\N}F_n}$ of $Y$ is clearly an element of $\mathcal F$ containing $Z$. Therefore it suffices to show that for every $n\in\N$, the set $A_{x,n}:=\{y\in F_x\colon(x,y)\in O_n\}$ is dense in $F_x$. So let us fix $n\in\N$. Let $U$ be an open subset of $Y$ intersecting $F_x$. Then $U$ intersect also $\bigcup_{n\in\N}F_n$, and so there are $i,j\in\N$ such that $f_{(i,j)}\in U$. Since condition (ii) immediately implies that the infinite sequence $(z_{(i,j,l)})_{l=1}^{\infty}$ is a $\sigma_{f_{(i,j)}}$-sequence, there is $l\ge n$ such that $z_{(i,j,l)}\in U$. By condition (iii), it holds $(x,z_{(i,j,l)})\in O_{i+j+l-2}\subseteq O_l\subseteq O_n$. At the same time, condition (i) implies that $z_{(i,j,l)}\in Z_{i+j+l-2}\subseteq\overline{F_{i+j+l-1}}\subseteq F_x$. Therefore $z_{(i,j,l)}\in A_{x,n}\cap F_x\cap U$, and so $A_{x,n}$ is dense in $F_x$.
\end{proof}

The next theorem is an improvement of~\cite[Theorem~4.4]{Lin_Moors}. Its proof is the same as the proof of~\cite[Theorem~4.4]{Lin_Moors}, we only need to use Lemma~\ref{lem:productOfBaireSpaces} instead of~\cite[Theorem~4.3]{Lin_Moors}.

\begin{theorem}\label{th:productOfBaireSpaces}
Let $(X,\tau)$ be a Baire space, and let $(Y,\tau')$ be a $\widetilde W$-space which possesses a rich family $\mathcal F$ of Baire subspaces. Then $X\times Y$ is a Baire space.
\end{theorem}

\begin{proof}
Let $\{O_n\colon n\in\N\}$ be a countable family of open dense subsets of $X\times Y$, we need to show that $\bigcap_{n\in\N}O_n$ is dense as well. So let $U$ be a nonempty open subset of $X$ and $V$ be a nonempty open subset of $Y$, we will show that $\bigcap_{n\in\N}O_n\cap(U\times V)\neq\emptyset$. Fix an arbitrary point $z\in V$. By Lemma~\ref{lem:productOfBaireSpaces} used on $Z=\{z\}$, the set
\begin{equation}\nonumber
\begin{split}
\big\{x\in X\colon&\text{ there is }F_x\in\mathcal F\text{ containing }z\text{ such that }\\
&\{y\in F_x\colon(x,y)\in O_n\}\text{ is dense in }F_x\text{ for all }n\in\N\big\}
\end{split}
\end{equation}
is comeager in $X$, and so the (bigger) set
\begin{equation}\nonumber
\begin{split}
\big\{x\in X\colon&\text{ there is }F_x\in\mathcal F\text{ intersecting }V\text{ such that }\\
&\{y\in F_x\colon(x,y)\in O_n\}\text{ is dense in }F_x\text{ for all }n\in\N\big\}
\end{split}
\end{equation}
is also comeager in $X$. Therefore there are $x_0\in U$ and $F_{x_0}\in\mathcal F$ intersecting $V$ such that $\{y\in F_{x_0}\colon(x_0,y)\in O_n\}$ is dense in $F_{x_0}$ for all $n\in\N$. As $F_{x_0}$ is a Baire space, the set $\{y\in F_{x_0}\colon(x_0,y)\in\bigcap_{n\in\N}O_n\}$ is also dense in $F_{x_0}$. So there is $y_0\in F_{x_0}\cap V$ such that $(x_0,y_0)\in\bigcap_{n\in\N}O_n$. In particular, $\bigcap_{n\in\N}O_n\cap(U\times V)\neq\emptyset$.
\end{proof}

For our next application of $\widetilde W$-space we need to consider another game.

\medskip

The {\it Choquet} game, $\mathcal{G}(X)$, played on a topological space $(X,\tau)$ between two players $\beta$ (who starts the game) and $\alpha$ who alternately construct a decreasing sequence 
$(B_n)_{n\in\N}$ of nonempty open subsets of $X$. Player~$\alpha$ wins the game if $\bigcap_{n\in\N}B_n\not= \emptyset$. Otherwise player~$\beta$ wins.

\medskip

The importance of this definition is revealed next.

\begin{theorem}[{\cite[Theorem 8.11]{Kechris}}] A topological space $(X,\tau)$ is a Baire space if, and only if, the player $\beta$ does not have a winning strategy in the Choquet game 
played on $X$.
\end{theorem}

The final two notions required for our next theorem are that of separate continuity and quasi-continuity.
A function $g : X \times Y \to Z$ that maps from a product of topological spaces $(X,\tau)$ and $(Y,\tau')$ into a topological space
$(Z,\tau'')$ is said to be \emph{separately continuous} on $X \times Y$ if for each $x_0 \in X$ and $y_0 \in Y$ the functions $y \mapsto g(x_0, y)$ and $x \mapsto g(x, y_0)$ are both continuous on $Y$ and $X$ respectively.

Suppose that $f:X \to Y$ is a function and $x_0 \in X$.  Then we say that $f$ is {\it quasi-continuous at $x_0$} if for every open neighborhood
$U$ of $x_0$ and $W$ of $f(x_0)$ there exists a nonempty open subset $V$ of $U$ such that $f(V) \subseteq W$.

The next theorem is a variant of \cite[Theorem 3.1]{Mirmostafaee} which deals with $W$-points instead of $\widetilde W$-points.

\begin{theorem}\label{QC-result} Suppose that $(X,\tau)$, $(Y, \tau')$ and $(Z,\tau'')$ are topological spaces and $f:X\times Y \to Z$ is a separately continuous function. If $(X,\tau)$ is a Baire space, $(Z,\tau'')$ is a regular space and $y_0 \in Y$ is a $\widetilde W$-point, then $f$ is quasi-continuous at each point of $X \times \{y_0\}$.
\end{theorem}

\begin{proof} Suppose, in order to obtain a contradiction, that $f$ is not quasi-continuous at some point $(x_0,y_0) \in X \times \{y_0\}$.  Then there exists an open neighborhood $U_0$ of $x_0$, an open
neighborhood $V_0$ of $y_0$ and an open neighborhood $W$ of $f(x_0,y_0)$ such that $f(U' \times V') \not\subseteq \overline{W}$ for any pair of nonempty open subsets $U'$ of $U_0$ and $V'$ of $V_0$.
Since, $x \mapsto f(x,y_0)$, is continuous and $f(x_0,y_0) \in W$, we may assume, by possibly making $U_0$ smaller that $f(U_0 \times \{y_0\}) \subseteq W$.

\medskip

We will now inductively define a strategy $t$ for the player $\beta$ in the Choquet game $\mathcal{G}(X)$. Let $\sigma$ be a winning strategy for the player I in the $\widetilde{G}(y_0)$ played on $Y$.

{\it Step 1.} If $\sigma(\emptyset) \cap V_0 = \emptyset$ then choose $y_1 \in \sigma(\emptyset)$ - any choice is fine - and define $t(\emptyset) := U_0$.  Otherwise, choose $y_1 \in \sigma(\emptyset) \cap V_0$
and $x' \in U_0$ such that $f(x',y_1) \not\in \overline{W}$.  Since, $x \mapsto f(x,y_1)$, is continuous there exists an open neighborhood $U'$ of $x'$, contained in $U_0$ such that
$f(U' \times \{y_1\}) \subseteq Z\setminus \overline{W}$.  Define $t(\emptyset) := U'$.  

\medskip

Now suppose that $ y_j \in Y$ and $U_j$ have been defined for each $2 \leq j \leq n$ so that (i) $U_1, U_2, \ldots, U_n$ are the first $n$ moves of the player $\alpha$ in the $\mathcal{G}(U_0)$ played on $U_0$; \\ (ii)
either $\sigma(y_1, \ldots, y_{j-1}) \cap V_0 = \emptyset$, $y_j \in \sigma(y_1, \ldots, y_{j-1})$ and $t(U_0, \ldots, U_{j-1}) = U_{j-1}$ 

 or 
 
\noindent $\sigma(y_1, \ldots, y_{j-1}) \cap V_0 \not= \emptyset$, $y_j \in \sigma(y_1, \ldots, y_{j-1})\cap V_0$ and $f(t(U_0, \ldots, U_{j-1}) \times \{y_j\})$ is a subset of $Z \setminus \overline{W}$.

\medskip

{\it Step n+1}. If $\sigma(y_1, \ldots, y_{n}) \cap V_0 = \emptyset$ choose $y_{n+1} \in \sigma(y_1, \ldots, y_{n})$ and define $t(U_1, \ldots, U_{n}) := U_{n}$. Otherwise, choose $y_{n+1} \in \sigma(y_1, \ldots, y_{n})\cap V_0$
and $x' \in U_n$ such that $f(x',y_{n+1}) \not\in \overline{W}$. Since, $x \mapsto f(x,y_{n+1})$, is continuous there exists an open neighborhood $U'$ of $x'$, contained in $U_n$ such that $f(U' \times \{y_{n+1}\}) \subseteq Z \setminus \overline{W}$.
Define $t(U_1, \ldots, U_n) := U'$.

\medskip

This completes the definition of $t$. Since $(X,\tau)$ is a Baire space, $t$ is not a winning strategy for the player $\beta$ in the $\mathcal{G}(X)$ game. Hence there exists a play $\{U_n\}_{n=1}^\infty$ where
player $\alpha$ wins, i.e., $\bigcap_{n \in \N} U_n \not= \emptyset$. Also, $y_0$ is an accumulation point of the sequence $(y_n)_{n \in \N}$, since $\sigma$ is a winning strategy for the player I in the $\widetilde{G}(y_0)$ game.  Let $(n_k)_{k=1}^\infty$ be a strictly increasing
sequence of natural numbers such that $\{n \in \N: y_n \in V_0\} = \{n_k:k \in \N\}$. Then $y_0 \in \overline{\{y_{n_k}:k \in \N\}}$.  Let $x_\infty \in \bigcap U_n \subseteq U_0$.  Thus,
$f(x_\infty, y_{n_k}) \not\in \overline{W}$ for all $k \in \N$. However, this implies that $f(x_\infty, y_0) \not\in W$ since, $y \mapsto f(x_\infty, y)$, is continuous and $y_0 \in \overline{\{y_{n_k}:k \in \N\}}$. This contradicts our assumption that
$f(U_0 \times \{y_0\}) \subseteq W$. Hence, $f$ must be quasi-continuous at each point of $X \times \{y_0\}$.
\end{proof} 

This result has implications for semitopological groups.

\medskip

A triple $(G, \cdot, \tau)$ is called a {\it semitopological group} ({\it topological group}) if $(G, \cdot)$ is a group, $(G, \tau)$ is a topological space and the multiplication operation ``$\cdot$'' is
separately continuous on $G \times G$ (jointly continuous on $G \times G$ and the inversion mapping, $g \mapsto g^{-1}$,  is continuous on $G$). 

\medskip

Let $(X, \tau)$ be a topological space. Following E. Reznichenko, (see, \cite{moors}) we shall say that a subset $W \subseteq X \times X$ is {\it separately open, in the second variable,} if for each  $x \in X$, $\{z \in X: (x,z) \in W\} \in \tau$ and we shall say that a topological space $(X, \tau)$ is a {\it $\Delta$-Baire space} if for each separately open, in the second variable set $W$, containing $\Delta_X := \{(x,y) \in X \times X: x=y\}$, there exists a nonempty open subset $U$ of $X$ such that $U \times U \subseteq \overline{W}$.  Many spaces are $\Delta$-Baire spaces. Indeed, all metrisable Baire spaces and all locally \v Cech-complete spaces are $\Delta$-Baire spaces, see \cite{moors}.

\begin{corollary}\label{Corollary1} Let $(G, \cdot, \tau)$ be a semitopological group.  If $(G, \tau)$ is: (i) a regular  Baire $\widetilde W$-space and (ii) a  $\Delta$-Baire space, then $(G, \cdot, \tau)$ is a topological group. In particular, if $(G, \tau)$ is a metrisable Baire space, then $(G, \cdot, \tau)$ is a topological group.
\end{corollary} 
\begin{proof} This follows directly from Theorem \ref{QC-result} and Theorem 1 in \cite{moors} which states that a semitopological group that is a regular $\Delta$-Baire space and whose multiplication operation is quasi-continuous is a topological group.
\end{proof}

We shall end this paper with another application of Theorem \ref{th:productOfBaireSpaces} and Theorem \ref{QC-result}. To state this corollary we need to recall the following definition. Let $(Z,\tau)$ be a topological space and $\rho$ some metric on $Z$.
The space $(Z,\tau)$ is said to be {\it fragmented by the metric $\rho$}, if for every $\varepsilon >0$ and every nonempty subset $A$ of $Z$ there exists a nonempty relatively open subset $B$ of $A$ with $\rho-\mathrm{diameter}(B) < \varepsilon$. In such a case the
space $(Z,\tau)$ is called {\it fragmentable}.

\medskip

An important theorem concerning fragmentable spaces is given next.

\begin{theorem}[\mbox{\cite[Theorem 1]{KM}}] \label{continuity}Let $(X,\tau)$ be a Baire space and $f:X \to Z$ be a  quasi-continuous map from $(X,\tau)$ into a topological space $(Z,\tau')$ which is fragmented by some metric $\rho$. Then there exists a dense $G_\delta$-subset $C \subseteq X$ 
at the points of which $f:(X,\tau) \to (Z,\rho)$ is continuous. In particular, if the topology generated by the metric $\rho$ contains the topology $\tau'$, then $f:(X,\tau) \to (Z,\tau')$ is continuous at every point of the set ¨$C$.
\end{theorem} 

\begin{corollary}\label{Cor:sepCont-cont} Suppose that $(X,\tau)$, $(Y, \tau')$ and $(Z,\tau'')$ are topological spaces and $f:X\times Y \to Z$ is a separately continuous function. If: (i) $(X,\tau)$ is a Baire space;  (ii)  $(Y,\tau')$ is a $\widetilde W$-space which possesses a rich family $\mathcal F$ of Baire subspaces and
(iii) $(Z,\tau'')$ is a regular space that is fragmented by some metric $\rho$ whose topology contains the topology $\tau''$. Then $f$ is continuous at the points of a dense $G_\delta$-subset of $X \times Y$.
\end{corollary}
\begin{proof} From Theorem \ref{th:productOfBaireSpaces}, $X \times Y$ is a Baire space. From Theorem \ref{QC-result}, $f:X \times Y \to Z$ is quasi-continuous.  The result then follows from Theorem \ref{continuity}.
\end{proof}
The utility of this result stems from the fact that, in addition to all  metrisable spaces, there are many topological spaces $(Z,\tau)$ that are fragmented by some metric $\rho$ whose topology contains the original topology $\tau$, see \cite{KM2}. 

\bibliographystyle{elsarticle-num}
\bibliography{References}

\end{document}